\documentclass[a4paper,11pt]{amsart}
\usepackage{amssymb}
\usepackage{amsmath}
\usepackage{amsthm}
\usepackage{graphicx}

\newcommand{\fG}{\mathfrak{G}}
\newcommand{\fM}{\mathfrak{M}}
\newcommand{\fX}{\mathfrak{X}}

\newcommand{\cM}{\mathcal{M}}
\newcommand{\cA}{\mathcal{A}}
\newcommand{\cG}{\mathcal{G}}
\newcommand{\cP}{\mathcal{P}}
\newcommand{\cF}{\mathcal{F}}

\newcommand{\cV}{\mathcal{V}}

\newcommand{\mZ}{\mathbb{Z}}
\newcommand{\mR}{\mathbb{R}}
\newcommand{\mN}{\mathbb{N}}
\newcommand{\mF}{\mathbb{F}}

\DeclareMathOperator{\rg}{range}
\DeclareMathOperator{\dom}{dom}
\DeclareMathOperator{\card}{card}

\newcommand{\st}{\mid}
\newcommand{\bigst}{\bigm|}

\theoremstyle{plain}
  \newtheorem{lemma}{Lemma}[section]
  \newtheorem{theorem}[lemma]{Theorem}
  \newtheorem{proposition}[lemma]{Proposition}
  
  \newtheorem{definition}[lemma]{Definition}
          
\theoremstyle{definition}
  \newtheorem{remark}[lemma]{Remark}
  \newtheorem{example}[lemma]{Example}

\title{Suspensions of Bernoulli shifts}

\thanks{2010 {\it Mathematics Subject Classification}. Primary 37C85. Secondary 37B10, 57R30}

\thanks{Version date: July 17, 2013}

\author{\'Alvaro Lozano-Rojo}

\address{\'Alvaro Lozano-Rojo, Centro Universitario de la Defensa, Academia General Militar, Carretera Huesca s/n, E-50090 Zaragoza, Spain - and - IUMA, Universidad de Zaragoza.}
\email{alvarolozano@unizar.es}

\author{Olga Lukina}

\address{Olga Lukina, Department of Mathematics, Statistics and Computer Science, University of Illinois at Chicago, 322 SEO (M/C 249), 851 S. Morgan St., Chicago, IL 60607, USA}
\email{lukina@uic.edu}

\begin{document}

\maketitle

\begin{abstract}
  We show that for a given finitely generated group, its Bernoulli shift space 
  can be equivariantly embedded as a subset of a space of pointed trees with 
  Gromov-Hausdorff metric and natural partial action of a free group. Since the 
  latter can be realised as a transverse space of a foliated space with leaves 
  Riemannian manifolds, this embedding allows us to obtain a suspension of such 
  Bernoulli shift. By a similar argument we show that the space of pointed trees is 
  universal for compactly generated expansive pseudogroups of transformations.
\end{abstract}

\section{Introduction}
\label{intro}
Let $G$ be a semigroup with the identity element, that is, $G$ is a set equipped 
with an associative binary operation. Let $S$ be a finite set, and consider the 
set of all maps $\Sigma(G, S) = \{ \sigma:G\to S\}$. Give the set $\Sigma(G,S)$ 
the product topology in the standard way, and define an action of $G$ on 
$\Sigma(G, S)$ by
\[
 \sigma(g') \cdot g = \sigma(gg').
\]
Thus one obtains a dynamical system $(\Sigma(G,S), G)$, called the 
\emph{Bernoulli shift} \cite{Cnrt}. The most well-known cases are $G = \mZ$, 
$G = \mN$ and $G = \mZ^n$ (see \cite{LindSchmidt} for a survey). In the cases 
$\mZ$ and $\mN$ one can study shifts using graphs generated by finite automata. 
This article proposes a way to represent the dynamics of a shift system using 
graphs in the case when $G$ is any finitely generated group. 

We will do that by constructing an equivariant embedding of $(\Sigma(G,S), G)$ 
into a metric space of pointed trees with dynamics given by a pseudogroup 
action. This space is obtained as follows.

For a symmetric set $G^1$ of generators of $G$, consider a graph $\cG$ with 
unoriented labeled edges, which is essentially the Cayley graph of $G$ with 
respect to $G^1$. We then denote by $X$ the set of all non-compact subgraphs of 
$\cG$ with only trivial loops, that is, every such subgraph $T \subset X$ is a 
tree. We also require that the identity element $e$ is a vertex of $T$, and 
therefore $(T,e)$ is a pointed metric space with the standard length metric. 
Distances between elements of $X$ are measured with the help of the 
Gromov-Hausdorff metric $d_{GH}$, and it turns out that $(X,d_{GH})$ is a 
compact totally disconnected space 
\cite{Ghys1999,Blanc2001,LozanoSp,Lukina2011}. There is a natural partial action 
of the free group $F_n$ on $X$, where $n$ is the half cardinality of $G^1$, and 
this action gives rise to the action of a pseudogroup $\fG$ on $X$. One can 
suspend the action of $\fG$ on $X$ to obtain a smooth foliated space $\fM_G$ 
with $2$-dimensional leaves \cite{Ghys1999,Blanc2001,LozanoSp,Lukina2011}. By 
this construction, for $T \in X$ the corresponding leaf $L_T \subset \fM_G$ can 
be thought of as the two-dimensional boundary of the thickening of a graph of 
the orbit of $T$ under the action of $\fG$. Following \cite{Lukina2011}, we call 
the closure $\cM = \overline{L}$ of a leaf $L \subset \fM_G$ a \emph{graph 
matchbox manifold}. The term `matchbox manifold' refers to the fact that 
$(X,d_{GH})$ is a totally disconnected space, and stems from the study of flows 
on $1$-dimensional continua \cite{AO1995,AM1988}. In the latter case a foliation 
chart is homeomorphic to the product of an open interval and a zero-dimensional 
space, and so can be thought of as a box of matches, where each match 
corresponds to a path-connected component of the chart.

The construction of a foliated space transversely modeled on the space of 
pointed trees with pseudogroup action was introduced by R.~Kenyon and used by 
\'E.~Ghys \cite{Ghys1999} in the case $G = \mZ^2$ to obtain an example of a 
space foliated by Riemannian manifolds such that each leaf is dense and there are 
leaves with different conformal types. Blanc \cite{Blanc2001} considered a 
similar construction for $G = F_n$, a free group on $n$ generators, and the 
first author \cite{LozanoSp}, also with F.~Alcalde Cuesta and M.~Macho Stadler 
\cite{ALM2008}, studied this construction for the case of an arbitrary finitely 
generated group. In the case $G = \mZ^2$ the first author \cite{Lozano2009} 
found examples with interesting ergodic properties. Blanc \cite{Blanc2001} found 
an example of a graph matchbox manifold with specific asymptotic properties of 
leaves. 

A systematic study of the dynamical and topological properties of $\fM_G$ was 
done by the second author in \cite{Lukina2011}. In particular, in 
\cite[Theorem~1.3]{Lukina2011} the second author studied a partial order on 
$\fM_G$ given by inclusions, which is equivalent to the study of the orbit 
structure of the pseudogroup dynamical system $(X,\fG)$. Using the notion of a 
level of a leaf, initially introduced for codimension $1$ foliations by 
J.~Cantwell and L.~Conlon \cite{CC1} (see also 
\cite{CC3,Hector2,Hector3,Nish1977,Nish1979}) she constructed hierarchies of 
graph matchbox manifolds at infinite levels. The results of \cite{Lukina2011} 
also show that the orbit structure of $(X,\fG)$ is reminiscent of the orbit 
structure of Bernoulli shifts, for example there is a meager subset of points 
with finite orbits, and a residual subset of points with dense orbits. 
Therefore, it is natural to seek the relation between the two. The second author 
has been asked this question by M.~Barge, S.~Hurder and A.~Clark, when 
presenting her work at conferences. Our main theorem gives the answer to this 
question.

Write the symmetric generating set $G^1$ as $G^{1+}\sqcup G^{1-}$ where $G^{1+}$ 
is the \emph{positive} set of generators and $G^{1-} = (G^{1+})^{-1}$ is the 
\emph{negative} set.

\begin{theorem}
 \label{thm-main}
 Let $G$ be a finitely generated group, and $(\Sigma(G, S),G)$ be the Bernoulli
 shift. Given a set of generators $G^1$ of $G$, and an injective map
 $\alpha: G^{1+} \times S \to F_n^{1+}$, there exists an embedding
 \[
   \Phi_\alpha : \Sigma(G, S) \to X_n,
 \]
 which is an orbit equivalence. Moreover, for any $\sigma\in\Sigma(G, S)$ and
 $g\in G^{1+}$
 \[
   \Phi_\alpha(\sigma \cdot g)
       = \Phi_\alpha(\sigma) \cdot \alpha(g, \sigma(e)).
 \]
 and if $g^{-1} \in G^{1+}$ then
 \[
   \Phi_\alpha(\sigma \cdot g)
     = \Phi_\alpha(\sigma) \cdot
             \bigl[\alpha(g^{-1}, \sigma(g)) \bigr]^{-1}.
 \]
\end{theorem}

Since $X_n$ can be realised as a transverse space of a foliated space with 
leaves Riemannian manifolds, this embedding allows us to obtain a suspension of 
a Bernoulli shift for arbitrary $G$ and $S$.

Recall \cite{Cnrt} that a dynamical system $(\Omega, G)$, where 
$(\Omega,d_\Omega)$ is a metric space, is $\epsilon$-expansive if for any
$x, y \in \Omega$ there exists $g \in G$ such that
$d_{\Omega}(g(x),g(y)) \geq \epsilon$. Bernoulli shifts provide
examples of $\epsilon$-expansive (for some $\epsilon$) dynamical systems. It is 
known \cite[Proposition 2.6]{Cnrt} that every expansive action of a group $G$ on 
a compact metrizable space $\Omega$ is a quotient of the $G$-action on a closed $G$-invariant subspace 
of $\Sigma(G,S)$ for some finite set $S$, i.e. of a subshift. If $\Omega$ is totally disconnected, then the quotient
map $\pi: \Sigma(G,S) \to \Omega$ can be made a conjugacy \cite[Proposition 2.8]{Cnrt}. This conjugacy is 
not unique, and there is no canonical way to choose the set $S$. 

As a consequence of Theorem~\ref{thm-main} we extend this result to 
pseudogroup dynamical systems $(\Omega, \Gamma)$, where $\Omega$ is a 
$0$-dimensional compact metrizable space and $\Gamma$ admits a \emph{compact 
generating system} \cite{haefliger1985,haefliger2002}. The latter property means 
that $\Omega$ contains a relatively compact open set $Y$ meeting all orbits of 
$\Gamma$ and the reduced pseudogroup $\Gamma|_Y$ is generated by a finite set 
$\Lambda_Y$ of elements of $\Gamma$ such that each $\gamma\in\Lambda_Y$ is the 
restriction of an element $\bar\gamma\in\Gamma$ and such that the closure of 
$\dom\gamma$ is contained in $\dom\bar\gamma$. Recall \cite{Hurder2010} that a 
pseudogroup $\Gamma$ on $\Omega$ is $\epsilon$\emph{-expansive} if for all 
$w\ne w'\in\Omega$ with $d(w,w')<\epsilon$ there exists $\gamma\in\Gamma$ with 
$w,w'\in\dom h$ such that $d(\gamma(w),\gamma(w'))\geq\epsilon$. 

\begin{theorem}
  \label{cor-pseudgroupuniv}
  Let $\Omega$ be a $0$-dimensional compact metrizable space, and $\Gamma$ be an 
  $\epsilon$-expansive pseudogroup of transformations on $\Omega$ which admits a 
  compact generating system. Then there is an embedding $\Omega\to X_n$ 
  equivariant with respect to the actions of $\Gamma$ and $\fG_n$, for a large 
  enough $n$.  
\end{theorem}

In other words, $(\Omega, \Gamma)$ is conjugate to the $\fG_n$-action on a closed $\fG_n$-invariant subset of $X_n$.

The rest of the paper is organized as follows. In Section~\ref{sec-prelim} we
recall definitions and basic properties of Bernoulli shifts, and sketch the
construction of a foliated space of graph matchbox manifolds. In
Section~\ref{sec-prooftheor} we present a proof of Theorem~\ref{thm-main}. 
Section~\ref{sec-univ} gives a proof of Theorem ~\ref{cor-pseudgroupuniv}, and 
in Section \ref{semigroup} we show that our method is applicable to the 
semigroup $G = \mN_0$, but not to general 
semigroups.

\section{Preliminaries}
\label{sec-prelim}

\subsection{Bernoulli shifts}
\label{Bern}

Let $\Gamma$ be a semigroup with the identity element, that is, $\Gamma$ is a 
set equipped with an associative binary operation, and there is $e \in \Gamma$ 
such that for every $\gamma \in \Gamma$, $\gamma e = e \gamma = \gamma$. Let $S$ 
be a finite set, and consider the set $\Sigma = \Sigma(\Gamma,S)$ of maps 
$\sigma: \Gamma \to S$, with the right action of $\Gamma$ given by
\begin{equation}
  \label{eq-action}
  \sigma(\gamma') \cdot \gamma = \sigma (\gamma \gamma').
\end{equation}
We give $S$ discrete topology, and $\Sigma$ the product topology, that is, given 
a finite set $F \subset\Gamma$, and an element $\sigma'\in\Sigma$, a basic open 
set is given by
\[
  C_{F,\sigma'}
    = \bigl\{\, 
          \sigma \in \Sigma 
          \bigst
          \sigma(\gamma) = \sigma' (\gamma) ~\text{for all}~ \gamma \in F 
      \,\bigr\}.
\]
Then $\Sigma$ with this topology is metrizable, compact and totally 
disconnected and,
if $\card S$ is strictly greater than $1$, perfect (see, for 
instance, \cite{Cnrt}).

\begin{definition}[\cite{Cnrt}]
  The space $\Sigma$ equipped with the action of $\Gamma$ given by  
  \eqref{eq-action} is called a \emph{Bernoulli shift}. A closed 
  $\Gamma$-invariant subspace of $\Sigma$ is called a \emph{subshift}.
\end{definition}

For basic properties of Bernoulli shifts we refer the reader to \cite{Cnrt}, 
stating here only that Bernoulli shifts are expansive, i.e. the following 
definition is satisfied.

\begin{definition}[\cite{Cnrt}]
  \label{defn-expansivedyn} 
  A dynamical system $(\Omega,\Gamma)$, where $(\Omega,d_\Omega)$ is a metric 
  space, is $\epsilon$-expansive if for any $x, y \in \Omega$ there exists 
  $\gamma \in \Gamma$ such that $d_{\Omega}(\gamma(x),\gamma(y)) \geq \epsilon$.
\end{definition}

\begin{example}
  Let $S = \{0,\ldots,N-1\}$, and let $\Gamma = \mZ$ (resp. $\Gamma = \mN_0$). 
  In this case $\Sigma$ can be thought of as a set of $2$-ended (resp. 
  $1$-ended) sequences, and the action of $\Gamma$ is given by
  \[
    \sigma \cdot n = (\ldots \sigma (-1) .\sigma(0) \sigma(1) \ldots) \cdot n 
                   = (\ldots \sigma (n-1) .\sigma(n) \sigma(n+1) \ldots)
  \]
  if $\Gamma = \mZ$, and if $\Gamma = \mN$, it is given by
  \[
    \sigma \cdot n = (\sigma(0) \sigma(1) \ldots) \cdot n 
                   = (\sigma(n) \sigma(n+1) \ldots).
  \]
  A metric $d_\Sigma$ on $\Sigma$ compatible with the topology is given, for 
  example, by
  \[
    d_{\Sigma}(\sigma,\sigma') =
      \sum_{i=-\infty}^\infty
        \frac{1}{2^{\lvert i\rvert}} \lvert\sigma(i) - \sigma'(i) \rvert,
  \]
  if $\Gamma = \mZ$, and similarly in the case $\Gamma = \mN_0$, with summation 
  starting from $i = 0$. Then for $\epsilon < 1/2$ and two distinct elements 
  $\sigma,\sigma'$ there exists $\lvert n\rvert >0$ such that 
  $\sigma (n) \ne \sigma'(n)$, and $\tilde{\sigma}, \tilde{\sigma}'$ such that 
  $\tilde{\sigma} = \sigma - n$, and $\tilde{\sigma}' = \sigma' - n$, that is, 
  $\tilde{\sigma}(0) = \sigma(n) \ne \sigma'(n) = \tilde{\sigma}'(n)$, which 
  shows that $d_\Sigma(\tilde{\sigma},\tilde{\sigma}') > \epsilon$, and 
  $(\Sigma,\Gamma)$ is $\epsilon$-expansive. 
\end{example}  


\subsection{Graph matchbox manifolds}
\label{subsec-GMM}

We give a brief outline of the construction of Kenyon and Ghys. The content of 
this section can also be found in \cite{Lukina2011,Ghys1999,LozanoSp,ALM2008}. 

\subsubsection{Space of pointed trees with Gromov-Hausdorff metric}
\label{subsec-GH} 

Let $V(T)$ be the set of vertices, and $E(T)$ be the set of edges of a graph 
$T$. An edge $w \in E(T)$ may be given an orientation by specifying its starting 
and its ending vertex, denoted by $s(w)$ and $t(w)$ respectively. If we are not 
interested in the orientation of a directed edge, or if an edge is undirected, 
then we denote the set of its vertices by $V(w)$. 

\begin{figure}

  \includegraphics{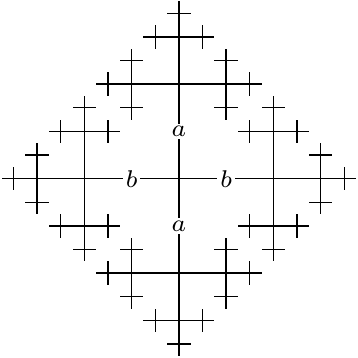}
  \caption{The Cayley graph $\cF_2$ of the free group $F_2=\langle a,b\rangle$.}
  \label{fig:cF_2}
\end{figure}

A subgraph of $T$ is a graph $T'$ with the set of vertices 
$V(T') \subseteq V(T)$ and the set of edges $E(T') \subset E(T)$. Edges of a 
graph $T$ are labeled by symbols contained in a set $\cA$, and sometimes we will 
use the notation $w_{a}$ for an edge $w \in E(T)$ labeled by $a \in \cA$.

Let $G$ be a group with a symmetric set of generators $G^1$, that is, 
$(G^1)^{-1} = G^1$, and $G^1$ does not contain the identity element. We consider 
the action of $G$ on itself on the right. Let $\cG'$ be the Cayley graph of $G$ 
with respect to the set of generators $G^1$; more precisely, set $V(\cG') = G$, 
and to each pair $g_1,g_2 \in V(\cG')$ such that $g_1 h = g_2$ for some 
$h \in G^1$, associate an edge $w_h \in E(\cG')$ with $s(w_h) = g_1$ and 
$t(w_h) = g_2$. Since $G^1$ is a symmetric generating set, the Cayley graph 
$\cG'$ is \emph{undirected}, that is, $g_1,g_2 \in V(\cG')$ are joined by an 
edge $w_h$ if and only if they are joined by an edge $w_{h^{-1}}$. We divide 
$G^1$ into a \emph{positive} part $G^{1+}$ and a \emph{negative} part $G^{1-}$, 
that is, if $h \in G^{1+}$, then $h^{-1}\in G^{1-}$. This division is, of 
course, arbitrary.

We now consider a graph $\cG$ such that $V(\cG') = V(\cG)$, and two vertices 
$g_1$ and $g_2$ are joined by an edge if and only if they are joined by an edge 
in $\cG'$. Thus for a pair of edges $w_h$ and $w_{h^{-1}}$ in $\cG'$ joining the 
same vertices we obtain a single unoriented edge in $\cG$ labeled by 
$h \in G^{1+}$.

We assume that $\cG$ is given the length structure $\ell$ in the usual way, so 
that every edge in $\cG$ has length $1$. Then the complete length metric $D$ on 
$\cG$ is defined by 
\[
  D(x,y) = \inf_\delta \bigl\{\,
              \ell(\delta) \bigst
              \delta:[0,1] \to \cG,\, \delta(0) = x,\, \delta(1) = y
           \,\bigr\},
\]
and $(\cG,D)$ becomes a \emph{length} metric space.

Consider the set $X$ of all subgraphs of $\cG$ which are non-compact, connected 
and simply connected, and which also contain the identity $e \in G$. A tree 
$T\in X$ is a subset of $\cG$, and therefore has an induced length structure 
$\ell$. 
The length structure $\ell$ gives rise to the length metric $d$ on 
$T$, and the pair $(T,e) \in X$ with metric $d$ becomes a pointed metric space.

The distance between pointed metric spaces $(T,e),(T',e) \in X$ is measured with 
the help of the Gromov-Hausdorff metric $d_{GH}$ \cite{Burago} (here for the 
purpose of this metric we have to restrict to maps preserving labeling of 
edges), and the set $X$ has metric topology induced from this metric. 
It is also customary to consider the 
\emph{box metric} $d_X$ on $X$, which gives this space the same topology as 
$d_{GH}$, and which is easier to work with. To introduce this metric we first 
need the following definition.

\begin{definition}
  Let $(A,g_1)$ and $(B,g_2)$ be two pointed (not necessarily non-compact) 
  subgraphs of $\cG$ with induced metric $d$. We say that $(A,g_1)$ and 
  $(B,g_2)$ are \emph{isomorphic} if there exists an isometry 
  $(A,g_1) \to (B,g_2)$ which maps $g_1$ into $g_2$ and preserves the labeling 
  of edges.
\end{definition}

Then for a pair of pointed metric spaces $(T,e),(T',e) \in X$ we define
\begin{equation}
  \label{eq-metric1}
  d_{X} (T,T') = e^{-r(T,T')},
\end{equation}
where
\[
  r(T,T') = \max\bigl\{
                r \in \mN \cup \{0\} \bigst 
                \exists\, B_T(e,r) ~\textrm{and}~  B_{T'}(e,r)~ 
                                    \textrm{are isomorphic}
            \bigr\}.
\]
Using the box metric, one can prove \cite{Ghys1999,LozanoSp,Lukina2011} that $(X,d_X)$ is 
totally disconnected, compact, and to 
derive a criterion for $(X,d_X)$ to be a perfect 
space. In particular, if $G = F_n$, a free 
group on $n$ generators with $n>1$, then the space of pointed trees $(X_n,d_X)$ 
is perfect and, therefore, is a Cantor set.


\subsubsection{Pseudogroup action on the space of pointed trees}
\label{subsubsec-actiongraphs}

Recall \cite{haefliger1985} that a \emph{pseudogroup of transformations on a 
space $\Omega$} is a family of homeomorphisms from open sets to open sets in 
$\Omega$ closed under restriction to open sets, composition, inversion and 
extension (gluing finitely many homeomorphisms to form another homeomorphism). 
We now define a pseudogroup action on $X$.

Let $\cP_e(T)$ be the set of paths $\delta:[0,1] \to T$ such that 
$\delta(0) = e$, $\delta(1) = g \in V(T)$ and $\delta$ is the shortest path 
between $e$ and $g$ in $T$. It follows that $\delta$ does not have 
self-intersections, and the image of $\delta$ in $T$ is the union of edges 
\[
  w_{h_{i_1}}\cup w_{h_{i_2}}\cup \cdots \cup w_{h_{i_m}}, ~ 
                    \text{where $h_{i_k} \in G^{1+}$ for $1 \leq k \leq m$.}
\]
Thus $\delta$ defines a word 
$\tilde{h}_{i_1} \tilde{h}_{i_2} \cdots \tilde{h}_{i_m} \in F_n$, which is 
composed as follows. Denote by $g_{k_1}$ and $g_{k_2}$ vertices of $w_{h_{i_k}}$ in 
such a way that $d_T(e,g_{k_1}) < d_T(e,g_{k_2})$, that is, $g_{k_1}$ is `closer' to the 
identity than $g_{k_2}$. Then set
\[
  \begin{cases}
    \tilde{h}_{i_k} = h_{i_k},      &   \text{if $g_{k_2} = g_{k_1} h_{i_k}$},      \\
    \tilde{h}_{i_k} = h_{i_k}^{-1}, &   \text{if $g_{k_2} = g_{k_1} h_{i_k}^{-1}$.}
  \end{cases}
\]
Then $\tilde{h}_1 \tilde{h}_2 \cdots \tilde{h}_n$ gives a representation of 
an element in $G$ in terms of the set of generators $G^1$, and each distinct 
path defines a distinct element. Therefore, the described procedure defines a 
well-defined map
\[
  p : \cP_e(T) \to F_n :
    \delta \mapsto \tilde{h}_{i_1} \tilde{h}_{i_2} \cdots \tilde{h}_{i_m}.
\]

\begin{definition}[\cite{Lukina2011}]
  \label{defn-actionany}
  Let $n < \infty$ be the cardinality of a set $G^{1+}$ of generators of $G$, 
  and $(X,d_X)$ be the corresponding set of pointed trees. Let $g \in F_n$. Then
  \begin{enumerate}

    \item $(T,e) \cdot g$ is defined if and only if there exists a path 
      $\delta \in \cP_e(T)$ such that $p(\delta) = g$.

    \item $(T',e) = (T,e) \cdot g$ if and only if $(T,g)$ and $(T',e)$ are 
      isomorphic. 

  \end{enumerate}
\end{definition}

Now for $r>0$ denote by $D_{X}(T,r)$ an open and closed subset of diameter $e^{-r}$ about $(T,e)$, that is,
\[
  D_{X}(T,r) = \bigl\{\, (T',e) \bigst d_{X}(T,T') \leq e^{-r} \,\bigr\}.
\]
For each $g \in F_n$ let $\ell_g = d_{\cF_n}(e,g)$, and form the union 
of clopen subsets
\begin{equation}
  \label{domgamma}
  D = \bigcup \bigl\{\,
            D_X(T,\ell_g) \bigst
            \text{$T \in X$, $\delta \in \cP_e(T)$ such that $p(\delta) = g$} 
      \,\bigr\}.
\end{equation}
The set $D$ is clopen since $G^1$ is a finite set and so \eqref{domgamma} is a 
finite union. The action of $g$ is defined on $D$, so we can define a map
\[
  \tau_g : D \to X_n : (T,e) \mapsto (T,e)\cdot g
\]
which is a homeomorphism onto its image. Then 
\[
  \fG = \langle\, \tau_g \st g \in F_n \,\rangle
\]
is a pseudogroup of local homeomorphisms with a generating set 
$\fG^1 = \{\,\tau_g\st g\in G^1\,\}$.


\subsubsection{Suspension of the pseudogroup action on the space of pointed trees}
\label{subsec-folRS}

The pseudogroup dynamical system $(X,\fG)$ can be realised as the holonomy 
system of a smooth foliated space $\fM_G$ as in the following theorem.

\begin{theorem}[\cite{Ghys1999,LozanoSp,Lukina2011}]
  \label{thm-genconstruction}
  Let $G$ be a finitely generated group, and $(X,d_X)$ be the corresponding 
  space of pointed trees with the action of a pseudogroup $\fG$. Then there 
  exists a compact metric space $\fM_G$, and a finite smooth foliated atlas 
  $\cV = \{\phi_i:V_i \to U_i \times \fX_i\}_{1 \leq i \leq \nu}$, where 
  $U_i \subset \mR^2$ is open, with associated holonomy pseudogroup $\Gamma$, 
  such that the following holds. 
  \begin{enumerate}

    \item The leaves of $\fM_G$ are Riemann surfaces.

    \item There is a homeomorphism onto its image
      \[
        t: X \to \bigcup_{1 \leq i \leq \nu} \fX_i,
      \]
      such that $t(X)$ is a complete transversal for the leaves of the 
      foliation, and $\Gamma|_{\tau(X)} = t_*\fG$, where $t_*\fG$ is the 
      pseudogroup induced on $t(X)$ by $\fG$.
\end{enumerate}
\end{theorem}

We refer for a proof to \cite{Ghys1999,LozanoSp,Lukina2011}, and only outline 
the idea of the construction here. Namely, $X$ can be covered by a finite number 
of clopen sets $D_{X}(\mathbf{a},1)$ of radius $1$, where $\mathbf{a}$ is a 
subgraph of a compact ball of radius $1$ in $\cG$. To each of these sets we 
associate a compact surface with boundary $\Sigma_\mathbf{a}$ which can be 
thought of as the two-dimensional boundary of the thickening of $\mathbf{a}$, 
and is homeomorphic to a $2$-sphere with at most $4$ open disks taken out. We 
then parametrise the surfaces near their boundaries, and obtain a foliated space 
\begin{equation}
  \label{ec:cayley_pseudo}
  \fM_G = \bigsqcup_{\mathbf{a} \in \mathbf{A}}
              D_X(\mathbf{a},1) \times \Sigma_\mathbf{a}/\!\!\sim
\end{equation}
by imposing an appropriate equivalence relation on the disjoint union 
$\bigsqcup_{\mathbf{a}\in\mathbf{A}} D_X(\mathbf{a},1) \times \Sigma_\mathbf{a}$. 
Namely, this equivalence relation `glues' the surfaces $\Sigma_\mathbf{a}$ along 
the parametrised regions near their boundaries. The resulting foliated space 
$\fM_G$ has Riemannian leaves and is modeled transversely on the totally 
disconnected space $X$.

\begin{definition}[\cite{Lukina2011}]
  Let $G$ be a finitely generated group, and $\fM_G$ be a suspension of 
  $(X,\fG)$ as in Theorem \ref{thm-genconstruction}. Then a \emph{graph matchbox 
  manifold} is the closure $\cM = \overline{L}$ of a leaf $L$ in $\fM_G$.
\end{definition}


\subsubsection{Properties of graph foliated spaces}

Some general properties of graph matchbox manifolds were shown in 
\cite{Lukina2011}. The first property we mention is the \emph{universal 
property}, which allows to reduce the study of all variety of foliated spaces 
$\fM_G$ to the cases $G = F_n$, a free group on $n$ generators. If $G = F_n$, we 
denote the space of pointed trees by $X_n$, the pseudogroup by $\fG_n$ and the 
corresponding suspension by $\fM_n$.

\begin{proposition}[\cite{Lukina2011}]
  \label{prop-univprop}
  Given a group $G$ with a set of generators $G^1 = G^{1+}\cup G^{1-}$, 
  $\card G^{1+}\leq n$, there exists a foliated embedding 
  \[
    \Phi:\fM_G \to \fM_n,
  \]
  where $\fM_G$ and $\fM_n$ are suspensions of $(X,\fG)$ and $(X_n,\fG_n)$ as in 
  Theorem~\ref{thm-genconstruction}.
\end{proposition}


\subsubsection{Bernoulli shifts versus graph matchbox manifolds}

As it has already been mentioned in Section~\ref{intro}, the question about the 
relation between transverse dynamics of the space of graph matchbox manifolds 
and Bernoulli shifts is a natural one, as the following properties of graph 
matchbox manifolds show. Recall that the definition of an $\epsilon$-expansive 
dynamical system can be generalised for pseudogroup dynamical systems as 
follows.

\begin{definition}[\cite{Hurder2010}]\label{def-exppseudo}
  The pseudogroup $\Gamma$ on $\Omega$ is $\epsilon$\emph{-expansive} if for all 
  $w\ne w'\in\Omega$ with $d(w,w')<\epsilon$ there exists $\gamma\in\Gamma$ with 
  $w,w'\in\dom \gamma$ such that $d(\gamma(w),\gamma(w'))\geq\epsilon$. 
\end{definition}

Bernoulli shifts $\Sigma(G,S)$ are expansive \cite{Cnrt}, and it was shown in 
\cite{Lukina2011} that $(X_n,\fG_n)$ is expansive. Besides, \cite{Lukina2011} 
proves that that the set of periodic orbits in $(X_n,\fG_n)$ (equivalently, the 
set of compact leaves in $\fM_n$) is dense, and there are points with dense 
orbits in $(X_n,\fG_n)$ (equivalently, dense leaves in $\fM_n$). Therefore, it 
is natural to study the relation between Bernoulli shifts (see 
\cite{KatokHasselblatt} on dynamics of shifts) and pseudogroup 
dynamical systems $(X_n,\fG_n)$.


\section{Embedding shift spaces into the space of pointed trees}
\label{sec-prooftheor}

We now give a proof of Theorem~\ref{thm-main}, that is, for a finitely generated 
group $G$, a finite set of symbols $S$ and a given injective map 
$\alpha:G^{1+} \times S\to F_n^{1+}$, where $G^{1+}$ denotes a `positive' part 
of a symmetric set of generators $G^1$ of $G$, there exists an embedding with 
appropriate equivariance properties 
\[
  \Phi_\alpha : \Sigma(G,S) \to X_n,
\]
for $n$ large enough.

Let $G = \langle\, G^{1} \st R \,\rangle$ be a presentation of $G$, where $G^1$ 
is a finite symmetric set of generators, and $R$ is a set of relations, possibly 
infinite. Denote by $M$ the cardinality of $G^{1+}$ and let 
$f: F_M \to F_M/R \cong G$ be a homomorphism. Then there is an induced map 
$\hat{f} : \Sigma(G,S) \to \Sigma(F_M,S)$ between the corresponding shift spaces 
given by
\[
  \hat{f}(\sigma)(w) = \sigma(f(w)), ~~ w \in F_M.
\]
The map $\hat{f}$ is continuous, indeed, for a finite subset $F\subset F_M$ and 
a $\sigma' \in\Sigma(F_M,S)$ we have
\[
  \hat{f}^{-1} (C_{F,\sigma' }) = 
    \bigl\{\, 
        \sigma \in\Sigma(G,S) \bigst 
                \sigma(f(w)) = \sigma'(w)~ \textrm{for any}~ w\in F 
    \,\bigr\}.
\]
The map $\hat{f}$ is also injective: given two different elements 
$\sigma, \sigma' \in \Sigma(G,S)$ there is $g\in G$ such that 
$\sigma(g) \neq \sigma'(g)$, and therefore 
$\hat{f}(\sigma) \neq \hat{f}(\sigma')$ as they differ on the set $f^{-1}(g)$. 
Since Bernoulli shift spaces are compact and Hausdorff, $\hat{f}$ is an 
embedding. Finally, since $f$ is a homomorphism, $\hat{f}$ is equivariant, that 
is,
\[
  \bigl[\hat{f}(\sigma) \cdot \gamma \bigr](w) = \hat{f}(\sigma)(\gamma w) = \sigma(f(\gamma w)) = \bigl[\sigma\cdot f(\gamma)\bigr](f(w))
                = \hat{f} [ \sigma \cdot f(\gamma) ](w)
\]
for any $w\in F_M$. This argument shows that to prove Theorem~\ref{thm-main}, it 
is enough to show that the embedding exists for free groups. This embedding will 
depend on the choice of the map $\alpha$, as in the proposition below.

\begin{proposition}
  \label{prop-freegroup}
  Let $F_M$ be a free group. Given an injective map 
  $\alpha:F_M^{1+} \times S\to F_n^{1+}$, there exists an embedding 
  \[
    \Phi_\alpha : \Sigma(F_M,S) \to X_n,
  \]
  equivariant with respect to the actions of $F_M$ and $F_n$. More precisely, 
  for any $\sigma \in \Sigma(F_M,S)$ and any $h \in F_M^{1+}$ we have
  \[
    \Phi_\alpha(\sigma \cdot h)
              = \Phi_\alpha(\sigma) \cdot \alpha(h, \sigma(e)),
  \]
  and if $h^{-1} \in F_M^{1+}$ then 
  \[
    \Phi_\alpha(\sigma \cdot h)
        = \Phi_\alpha(\sigma) 
                  \cdot \bigl[\alpha(h^{-1}, \sigma(e))\bigr]^{-1}.
  \]
\end{proposition}

\begin{proof}
  Given $\sigma\in\Sigma(F_M,S)$, we construct $\Phi_\alpha(\sigma) = T_\sigma$ 
  as a subgraph of $\cF_n$ by induction. 
  
  For $j \geq 1$, let $F_M^j$ be the set of words in $F_M$ of length $j$. Also set $F_M^0 = \{e\}$. The set of words of length less or equal to 
  $j$ is denoted by $F_M^{\leq j} = \bigcup_{0\leq i\leq j} F_M^i$. Define  
  $K_0 = \{e\}\subset\cF_n$ to be a graph consisting of just one vertex $e$ and 
  no edges. Finally we define 
  $\kappa_0:F_M^{\leq 0} = F_M^0 \to V(K_0)\subset V(\cF_n)$ in the only 
  possible way.

  We proceed by induction. We assume that we are given the following data:
  \begin{enumerate}

    \item A connected subgraph $K_{j-1} \subset \cF_n$ containing $e$.

    \item A bijective map $\kappa_{j-1} : F_M^{\leq j-1} \to V(K_{j-1})$ such 
      that if $w_k\in F_M^{\leq j-1} $, $k=1,2$, and $w_2 = w_1 h$ for some 
      $h\in F_M^{1+}$, then
      \[
        \kappa_{j-1}(w_2) = \kappa_{j-1}(w_1) \alpha(h,\sigma_1(e)),
      \]
      where $\sigma_k = \sigma \cdot w_k$, $k=1,2$.

  \end{enumerate}
  We now construct a set $F_M^{\leq j}$, a graph $K_j$ and a map 
  $\kappa_j:F_M^{\leq j} \to V(K_j)$. For any $w'\in F_M^{\leq j-1}$ set 
  $\kappa_j(w') = \kappa_{j-1}(w')$. Since $F_M$ is a free group, given 
  $w_2\in F_M^j$ there is a unique $w_1\in F_M^{j-1}$ such that 
  $w_2 = w_1 h$, and either $h \in F^{1+}_M$ or $h^{-1} \in F^{1+}_M$. In the 
  first case set
  \begin{equation}
    \label{ec:kappa_j1}
    \kappa_j(w_2) = \kappa_{j-1}(w_1) \alpha(h, \sigma_1(e)),
  \end{equation}
  and otherwise set
  \begin{equation}
    \label{ec:kappa_j2}
    \kappa_j(w_2) 
      = \kappa_{j-1}(w_1)
                \bigl[\alpha(h^{-1}, \sigma_2(e))\bigr]^{-1},
  \end{equation}
  where, as before, the juxtaposition denotes group multiplication in $F_n$. Let 
  $K_j$ be a subgraph of $\cF_n$ with the set of vertices 
  $\{\kappa_j(F_M^{\leq j})\}$. Define $T_\sigma = \bigcup_{j=0}^\infty K_j$. We 
  notice that the graph $K_j$ is in fact a closed ball $B_{T_\sigma}(e,j)$ of 
  radius $j$ in $K_j$.
  
  We show that $\Phi_\alpha$ is continuous. As the involved spaces are 
  metrizable it is enough to show sequential continuity. Let 
  $\{\sigma_k\} \to \sigma$, then (passing to a subsequence if necessary) there 
  exists an increasing sequence of integers $\{r_k\}$ such that
  \begin{equation}
    \label{ec:sigma_k}
    \sigma_\ell(w) = \sigma(w) \;
      \text{for all words $w\in F^{\leq k}_M$ and all $\ell \geq r_k$.}
  \end{equation}
  Now, by construction each tree is a union of finite graphs, that is, 
  $T_{\sigma_k} = \bigcup_{j\geq 0} K^k_j$, and 
  $T_\sigma=\bigcup_{j\geq 0} K_j$. Then \eqref{ec:sigma_k} implies that 
  $K^\ell_{k} = K_{k}$ for all $\ell \geq r_k$. Therefore, $\{T_{\sigma_k}\}$ 
  converges to $T_\sigma$, and $\Phi_\alpha$ is continuous.

  A continuous map between a compactum and a Hausdorff space is always closed, 
  and therefore, to show that $\Phi_\alpha$ is an embedding, it is enough to 
  show that it is injective. If $\sigma\neq\sigma'\in\Sigma(F_M,S)$, there is 
  $w\in F^{\leq j}_M$ for some $j\geq 0$ such that $\sigma(w)\neq\sigma'(w)$. 
  The definition of $\kappa_j$ in \eqref{ec:kappa_j1} and \eqref{ec:kappa_j2} 
  forces $K_j\neq K_j'$, where $T_\sigma = \bigcup_{j\geq0} K_j$ and 
  $T_{\sigma'} = \bigcup_{j\geq0} K'_j$. Hence $T_\sigma\neq T_{\sigma'}$.

  Finally, we show the equivariance of the embedding. We consider only the case 
  $h\in F_M^{1+}$, as the argument in the case $h^{-1} \in F_M^{1+}$ is similar. 
  Denote $\sigma' = \sigma \cdot h$. By the definition of the $F_M$-action 
  on $\Sigma(F_M,S)$ and by the construction there is map
  \[
    V(T_{\sigma'}) \to V(T_{\sigma}) :
          g \mapsto g\alpha(h, \sigma'(h^{-1}))^{-1},
  \]
  which means that $(T_{\sigma'},e)$ is isomorphic to 
  $(T_{\sigma}, \alpha(h, \sigma(e)))$. Then 
  \[
    T_{\sigma'} = T_\sigma \cdot  \alpha(h, \sigma(e)).
  \]
\end{proof}

\begin{remark}
  Given a graph $T_\sigma \in \rg(\Phi_\alpha)$, it is useful to know how to 
  recover $\sigma \in \Sigma (F_M,S)$. Therefore, we outline here the 
  construction of the inverse $\Phi_\alpha^{-1}$. We will denote by 
  $pr_i$ the $i^\text{th}$ coordinate projection.

  Let $T \in \rg\Phi_\alpha$, then every vertex in $T$ has degree $2M$, and the 
  labels of edges in $T$ are necessarily contained in 
  $\rg\alpha\subset F_{n}^{1+}$. Moreover, the following can be deduced from the 
  properties of free groups and the construction.
  
  Let $g \in V(T)$ be a vertex, $w_{h_1}$ and $w_{h_2}$ be edges such that 
  $\{g\} = V(w_{h_1}) \cap V(w_{h_2})$ and let $g_i \in  V(w_{h_i})\smallsetminus \{g\}$, 
  $i =1,2$, be another vertex. Suppose $g_i = g h_i$, $i=1,2$. Then
  \begin{equation}
    \label{eq:startvrule1}
    pr_2 \circ \alpha^{-1}(h_1) = pr_2 \circ \alpha^{-1}(h_2) = s \in S,
  \end{equation}
  and there is precisely $M$ edges with this property. 

  By construction $T = \bigcup_{j=1}^\infty K_j = B_T(e,j)$. We construct 
  $\sigma$ by induction.

  Set $F^0_M = \{e\} \in F_M$, then $\lambda_0 : V(K_0) = \{e\} \to F^0_M$ is 
  just the map of the identities. Let $E \subset E(T)$ be the set of $2M$ edges 
  adjacent to the identity $e$. For $w_h \in E$, let 
  $g \in V(w_h) \smallsetminus \{e\}$, 
  then either $g=h$ or $g = h^{-1}$ for $h \in F_n^{1+}$. In any case, 
  $h\in \rg\alpha$, and there is a unique pair $(\gamma,s) = \alpha^{-1}(h)$. If 
  $g = h^{-1}$, define $\lambda_1:V(K_1) \to F^{1-}_M$ and 
  $\sigma|_{F^{1-}_M}: F^{1-}_M \to S$ by
  \[
    \lambda_1(g) = \gamma^{-1}, ~ \textrm{and} ~ \sigma_1(\lambda_1(g)) = s.
  \]
  In the case $g = h$ set $\sigma(e) = s$, and $\lambda_1(g) = \gamma$. The 
  value of $\sigma(e)$ is well-defined by the property \eqref{eq:startvrule1}. 
  To define $\sigma$ on $F^{1+}_M$ we have to look at edges in 
  $B_T(e,2) \smallsetminus B_T(e,1)$.

  For every $g \in V(K_1)$ such that $g = h \in F^{1+}_n$, let 
  $E'\subset B_T(e,2)$ be the set of edges adjacent to $g$. By the properties of 
  $T$ there exists $w_{h'}$ such that for $g' \in V(w_{h'}) \smallsetminus \{g\}$ we have 
  $g' = gh'$. Let $(\gamma',s') = \alpha^{-1}(h')$, and set $\sigma(g) = s'$. 
  Then $\sigma(g)$ is well-defined by the property \eqref{eq:startvrule1}.

  To obtain the full map $\sigma$, proceed inductively. Given an injective map 
  $\lambda_{j-1}:V(K_{j-1})\to F_M^{\leq j-1}$, and $g \in V(K_{j-1}) \cap F^{j-1}_n$, 
  consider the set of edges adjacent to $g$ which do not lie in $K_{j-1}$. For 
  such an edge $w_h$ let $g'$ be the vertex in $V(w_h) \smallsetminus \{g\}$, and if 
  $g' = g h^{-1}$, define $\lambda_j(g') = \lambda_{j-1}(g) \left[ pr_1 \circ \alpha^{-1}(h)\right]^{-1}$ 
  and $\sigma(\lambda_j(g')) = pr_2 \circ \alpha^{-1}(h)$. If $g' = gh$, set 
  $\lambda_j(g') = \lambda_{j-1}(g) \left[ pr_1 \circ \alpha^{-1}(h)\right]$, and to determine 
  $\sigma(\lambda_j(g'))$ consider the set of edges adjacent to $g'$. Like at 
  the first step of the inductive procedure, there is an edge $w_{h'}$ such that 
  for $g'' = V(w_{h'}) \smallsetminus \{g'\}$ we have $g'' = g' h'$. Then let 
  $\sigma(\lambda_j(g')) = pr_2 \circ \alpha^{-1}(h')$.
\end{remark}
 
\begin{remark}
  Given a set of generators $G^{1+}$, the tree $T_\sigma$ is the universal 
  covering space of the graph of the orbit of $\sigma$ in $\Sigma(G,S)$.
\end{remark}

\begin{example} 
  Let $G = \mZ$ and $S=\{0,1\}$, then $\Sigma(\mZ,S)$ is the set of 
  bi-infinite sequences of $0$ and $1$. The positive generating set of $\mZ$ is 
  $\{1\}$. Consider the free group with $2$ generators $F_2$ with positive 
  generating set $\{a,b\}$. Define
  \[
    \alpha: \{1\} \times \{0,1\} \to \{a,b\}
  \]
  by $\alpha(1,0) = a$ and $\alpha(1,1) = b$.

  Let $\sigma : \mZ \to S$ be given by $\sigma(k) = 0$ if $k$ is even and 
  $\sigma(k) = 1$ if $k$ is odd. Then the orbit of $\sigma$ in $\Sigma(\mZ,S)$ 
  contain just two points, $\sigma$ itself and $\sigma' = \sigma + 1$. One 
  thinks of graphs $T_\sigma, T_{\sigma'} \subset \cF_2$ as `ladders' made up of 
  $a$ and $b$-edges. Each graph is a covering space of the graph of the orbit of 
  $T_\sigma$ in $X_2$, which consists of two vertices $T_\sigma$ and 
  $T_{\sigma'}$ joined by two edges, one marked by $a$ and another one marked by 
  $b$.
  \begin{figure}
    
    \includegraphics{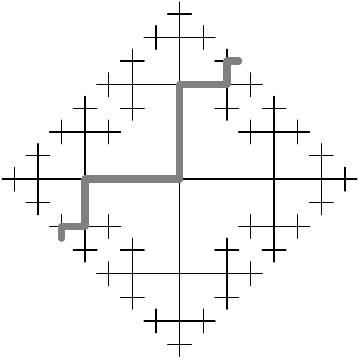}
   \caption{The tree $T_{\sigma}$}
    \label{fig:shiftZ}
  \end{figure}
\end{example}

\begin{remark}
  The embedding of Proposition~\ref{prop-freegroup} is highly non-unique. Given 
  a symmetric set of generators $F_M^1$ and a set $S$, the embedding depends on 
  the choice of $n \geq M \card S$, of a `positive' set $F_M^{1+}$ and of the 
  map $\alpha: F_M^{1+}\times S\to F_n^{1+}$. Moreover, the embedding of 
  $\Sigma(G,S)$ in Theorem~\ref{thm-main} depends on the choice of the 
  homomorphism $f:F_n\to G$.
\end{remark}

\begin{remark}
  Geometrically non-trivial relations in $G = \langle G^1 ~|~ R \rangle$ 
  manifest themselves as symmetries of graphs $T_\sigma$. For example, if 
  $G = \mZ^2$, then for every $g \in F_2$ and every $\sigma \in \Sigma(F_M,S)$ 
  we have $\sigma(g) = \sigma(aba^{-1}b^{-1}g)$, and for every $R>0$ the ball 
  $B_{T_\sigma}(e,R)$ is isometric to the ball 
  $B_{T_\sigma}(aba^{-1}b^{-1}, R)$, and similarly for other elements in $\mZ^2$ 
  congruent to $e$.
\end{remark}


\section{Universal space for compactly generated pseudogroups}
\label{sec-univ}

In order to prove Theorem~\ref{cor-pseudgroupuniv}, we adapt the proofs of 
\cite[Propositions 2.6, 2.8]{Cnrt} and Proposition~\ref{prop-freegroup}. We first recall some background facts.

In \cite{haefliger1985,haefliger2002} A. Haefliger showed that the holonomy 
pseudogroup of a compact foliated manifold satisfies a finiteness condition, called \emph{compact generation}.

\begin{definition}
  A pseudogroup of transformations $\Gamma$ of a space $\Omega$ is compactly 
  generated if $\Omega$ contains a relatively compact open set $Y$ meeting all 
  orbits of $\Gamma$, and the reduced pseudogroup $\Gamma|_Y$ is generated 
  by a finite set $\Lambda_Y$ of elements of $\Gamma$ such that each element 
  $\lambda\in\Lambda_Y$ is the restriction of an element 
  $\overline\lambda\in\Gamma$ with the closure of $\dom\lambda$ contained in 
  $\dom\overline\lambda$. Such generating system $\Lambda_Y$ is called a 
  \emph{compact generation system (CGS) of $\Gamma$ on $Y$}.
\end{definition}

In the totally disconnected case, it is possible to chose a CGS with compact 
domains, as the following lemma shows.
      
\begin{lemma}
  \label{lemma:CGS<->Full}
  Let $\Gamma$ be a pseudogroup of transformations of a $0$-dimensional locally 
  compact metrizable space $\Omega$. If $\Gamma$ is compactly generated there 
  exists a CGS $\Lambda_Y$ on a compact and open set $Y\subset\Omega$, where 
  $\dom\lambda$ is compact and open for each $\lambda\in\Lambda_Y$.
\end{lemma}

\begin{proof}
  By definition there exists a CGS $\Lambda_Z$ of $\Gamma$ on an open relatively 
  compact set $Z$ which meets all $\Gamma$-orbits. For each 
  $\lambda\in\Lambda_Z$, $\overline{\dom\lambda}$ is compact and contained in 
  the open set $\dom\overline\lambda$ for some $\overline\lambda\in\Gamma$. 
  By the hypothesis, the space $\Omega$ has a countable basis of 
  compact and open sets, and so we can cover $\overline{\dom\lambda}$ with a  
  finite number of compact open sets contained in $\dom\overline\lambda$. 
  Then their union $D_\lambda$ is again compact and open. We then define 
  $\lambda' = \overline\lambda|_{D_\lambda}$ which satisfies 
  $\lambda = \lambda'|_{\dom\lambda}$.

  Since $\Lambda_Z$ is finite, the set 
  $Y = \bigcup_{\lambda\in\Lambda_Z} D_\lambda$ is compact and open. Since 
  $Z\subset Y$, $Y$ meets all $\Gamma$-orbits. But in general the finite family 
  $\{\,\lambda' \st \lambda\in\Lambda_Z\,\}$ is not a generating set for 
  $\Gamma|_Y$, which we can amend for by adding a finite number of maps in the 
  following way. Since $Z$ meets every orbit, for each $y\in Y\smallsetminus Z$ 
  we can find $\gamma_y\in\Gamma$ such that its domain is compact, open and 
  contained in $Y$ and its range is contained in $Z$. Since $Y\smallsetminus Z$ 
  is compact, there is a finite family of maps $\{\gamma_i\}_{i\in I}$ such that 
  the union of their domains covers $Y\smallsetminus Z$.

  Finally, we claim that
  $\Lambda_Y = \{\,\lambda',\gamma_i\st\lambda\in\Lambda_Z, i\in I\,\}$ 
  is CGS of $\Gamma$ on $Y$. To see that, let $\gamma\in\Gamma|_Z$. We claim 
  that $\gamma$ locally belongs to $\langle\Lambda_Y\rangle$ and therefore it 
  belongs $\langle\Lambda_Y\rangle$ globally by properties of pseudogroups. 
  Indeed, if there is $z\in Z$  such that $\gamma(z)\in Z$ then this just 
  follows from the hypothesis. If $z \in Z$ and 
  $\gamma(z) = y\in Y\smallsetminus Z$, then there exists $\gamma_i$ such that 
  $\gamma_i(y)\in Z$ and locally $\gamma_i\circ\gamma\in\Gamma|_Z$ can be 
  written as a $\Lambda_Z$-word. Then $\gamma = \gamma_i^{-1}\circ w$ which is a 
  $\Lambda_Y$-word. The statement holds in the two other cases by a similar 
  argument.
\end{proof}

Recall from definition \ref{def-exppseudo} (see also \cite{Hurder2010}) that a 
pseudogroup $\Gamma$ on a Polish space $\Omega$ is $\epsilon$\emph{-expansive} 
if for all $w\ne w'\in\Omega$ with $d(w,w')<\epsilon$ there exists 
$\gamma\in\Gamma$ with $w,w'\in\dom h$ such that 
$d(\gamma(w),\gamma(w'))\geq\epsilon$, where $d$ is a metric on $\Omega$. We now 
prove Theorem~\ref{cor-pseudgroupuniv}, which shows that if $\Omega$ is compact, 
and $\Gamma$ is compactly generated and $\epsilon$-expansive, then the 
pseudogroup dynamical system $(\Omega, \Gamma)$ can be equivariantly embedded 
into the space of pointed trees $X_n$, for $n$ large enough. Together with 
Theorem~\ref{thm-main}, this result can be seen as a generalization of the 
well-known fact that every expansive dynamical system on a totally disconnected 
compact space is conjugate to a subshift \cite{Cnrt}. The proof of 
Theorem~\ref{cor-pseudgroupuniv} below uses the ideas of 
\cite[Propositions 2.6, 2.8]{Cnrt} in conjunction with the method of 
Proposition~\ref{prop-freegroup}.

\begin{proof}\emph{(of Theorem~\ref{cor-pseudgroupuniv}).}
  Let $\Gamma$ be a compactly generated pseudogroup on a compact metrizable 
  $0$-dimensional space $\Omega$ with the CGS $(Z,\Lambda_Z)$. By 
  Lemma~\ref{lemma:CGS<->Full} we can construct a CGS $(Y,\Lambda_Y)$ over a 
  clopen set $Y$ such that the domains of the elements of $\Lambda_Y$ are 
  compact, and we can use the same method to find a system of generators 
  $\Lambda_\Omega$ with the same properties on $\Omega$. We can also assume that 
  $\Lambda_\Omega$ is a symmetric set of generators, and denote by 
  $\Lambda_\Omega^+$ its `positive' part. 
  
  By assumption $\Gamma$ is $\epsilon$-expansive, and since $\Omega$ is 
  compact, 
  we can find a finite partition $\{B_s\}_{s \in S}$ of $\Omega$ by clopen 
  pairwise disjoint sets of diameter smaller than $\epsilon$. Let $m = \card S$.
  We now form the set $\mF$ of all formal finite words of elements 
  $\Lambda_\Omega$, that is, $\mF$ is a free group with 
  $M = \card\Lambda_\Omega^+$ generators. Then, given a word $g\in \mF$, there 
  is a pair $(\theta_g,\gamma_g)$, where $\gamma_g$ is a homeomorphism obtained 
  as the composition of generators in $g$, and $\theta_g$ is the clopen domain 
  of $\gamma_g$. Such compositions are not always defined, and we allow 
  $\theta_g$ to be an empty set. This is the main difference of the considered 
  situation from that in \cite[Propositions 2.6, 2.8]{Cnrt}. Let
  \[
    \widetilde{S} = S \cup \{s_\emptyset\},
  \]
  where $s_\emptyset$ would code the empty set, and let 
  $\Sigma(\mF, \widetilde{S}) = \{ \sigma : \mF \to \widetilde{S} \}$ be the 
  shift space. Define
  \[
    I(\sigma) = \bigcap_{g \in \mF, \theta_g \ne \emptyset} 
        \left\{ 
          \gamma_g^{-1}\left( B_{\sigma(g)} \cap \gamma_g(\theta_g) \right) 
        \right\},
  \]
  and let $\Psi \subset \Sigma(\mF, \widetilde{S})$ be the subset containing 
  those $\sigma$ for which $I(\sigma)$ is non-empty. Note that an arbitrarily 
  large number of elements in $\mF$ can be mapped by such a $\sigma$ to 
  $s_\emptyset$, which means that points in $I(\sigma)$ are not in the domain 
  of homeomorphisms corresponding to these elements. If $\theta_g = \emptyset$, 
  then for every $gh$, $h \in \mF$, we have $\theta_{gh} = \emptyset$, and so 
  $\sigma(gh) = s_\emptyset$.

  By a similar argument as for dynamical systems, $\epsilon$-expansivity of 
  $\Gamma$ and the choice of the covering $\{B_s\}_{s \in S}$ implies that each 
  $I(\sigma)$ is a single point. More precisely, let $x,y \in I(\sigma)$. Then 
  for any $g \in \mF$ with $\theta_g \ne \emptyset$ we have 
  $\gamma_g(x),\gamma_g(y) \in B_{\sigma(g)}$, which implies that 
  $d_\Omega(\gamma_g(x), \gamma_g(y)) < \epsilon$ for all $g \in \mF$, and so 
  $x = y$. It follows that the map
  \[
    \pi : \Psi \to \Omega: \sigma \mapsto I(\sigma)
  \]
  is well-defined and injective. However, $\Psi$ is not invariant under the 
  action of $\mF$. Indeed, it follows from definitions that if 
  $\sigma(g) \ne s_\emptyset$, then $I(\sigma \cdot g) = \gamma_g (I(\sigma))$ 
  and therefore $I(\sigma \cdot g) \in \Psi$. However, if 
  $\sigma(g) = s_\emptyset$, then there exists 
  $\sigma\cdot g\in \Sigma(\mF,\widetilde{S})$, while $\gamma_g(I(\sigma))$ is 
  not defined. We therefore define a \emph{partial action} of $\mF$ on $\Psi$ 
  by allowing $g \in \mF$ to act on $\sigma$ if and only if 
  $\sigma(g) \ne s_\emptyset$. Then the map $\pi$ is equivariant with respect 
  to this partial action of $\mF$ on $\Psi$ and the $\Gamma$-action on 
  $\Omega$. 
  
  We have to show that $\pi$ is continuous and surjective. Let 
  $\{\sigma_n\}\in\Psi$ be a sequence converging to $\sigma \in \Psi$. Then 
  for every $g \in \mF$ we can find an integer $n_g$ such that for all $n>n_g$ 
  we have $\sigma_n(g) = \sigma(g)$, that is, $\sigma_n(g)$ and $\sigma(g)$ are 
  in the same element of the partition. Since $\Omega$ is compact, by 
  extracting a subsequence, we can assume that $x_n = \pi(\sigma_n)$ converges 
  to an element $x \in \Omega$. If $\sigma(g) \ne s_\emptyset$, then for 
  $n > n_g$ we have $x_n\in \gamma_g^{-1}(B_{\sigma(g)})$, and 
  $x\in \gamma_g^{-1}(B_{\sigma(g)})$ since $\gamma_g^{-1}(B_{\sigma(g)})$ is 
  closed. If $\sigma(g) = s_\emptyset$, then for $n > n_g$ we have 
  $\sigma_n(g) = s_\emptyset$, that is, $x_n\notin\dom\gamma_g$. Since domains 
  of homeomorphisms in $\Gamma$ are open, their complements are closed and so 
  $x\notin\dom\gamma_g$. It follows that $x \in I(\sigma)$, and $\pi$ is 
  continuous. Surjectivity follows from the fact that the domains of 
  pseudogroup generators and their inverses form a clopen covering of $\Omega$. 
  Therefore $\pi$ is a homeomorphism.

  We now apply a procedure similar to the one in Theorem~\ref{thm-main} to 
  embed $\Psi$ into $X_n$ for a large enough $n$. The difference with 
  Theorem~\ref{thm-main} is that if $\sigma(g) = s_\emptyset$ we do not allow 
  $g$ to act on $\sigma$, and therefore we do not add any edge to the graph. As 
  before, we assume that a injective map 
  $\alpha:\Lambda_\Omega^+\times S\to F^{1+}_n$ has been chosen, for 
  $n\geq Mm$.

  We fix $\sigma\in\Psi$ and construct $T_\sigma$ as a subgraph of $\cF_n$ as 
  follows. We proceed by induction. Let $\{e\} = K_0 \subset T_\sigma$ be a 
  subgraph with a single vertex. As before, $\mF^j$ and $\mF^{\leq j}$ 
  represent the set of words of length $j$ and words of length less or equal to 
  $j$ in $\mF$. We also define
  \[
    \widetilde\mF^j = \bigl\{\, g\in\mF^j 
                                    \bigst \sigma(g)\neq s_\emptyset \,\bigr\},
  \]
  the set $\widetilde\mF^{\leq j}$ is defined in a similar way. We assume that 
  we are given the following data:
  \begin{enumerate}

    \item A connected subgraph $K_{j-1} \subset\cF_n$ containing $e$.

    \item A bijective map $\kappa_{j-1}:\widetilde\mF^{\leq j-1}\to V(K_{j-1})$ 
      on a subset $\widetilde\mF^{\leq j-1} \subset \mF^{\leq j-1}$, such that 
      if $g_k \in \widetilde\mF^{\leq j-1}$, $k=1,2$, and $g_2 = g_1 h$ for some 
      $h \in \Lambda_\Omega$, then 
      \[
        \kappa_{j-1}(g_2) = \kappa_{j-1}(g_1)\alpha(h,\sigma_1(e)),
      \]
      where $\sigma_k = \sigma\cdot g_k$.

  \end{enumerate}
  Now we construct the graph $K_j$ and the map 
  $\kappa_j:\widetilde\mF^{\leq j} \to V(K_j)$. For any 
  $g'\in \widetilde\mF^{j-1}$ set $\kappa_j(g') = \kappa_{j-1}(g')$. Since $\mF$ 
  is a free group, given $g_2\in\widetilde\mF^j$ there is a unique 
  $g_1\in\widetilde\mF^{j-1}$ such that $g_2 = g_1 \cdot h$, and either 
  $h\in\Lambda_\Omega^+$ or $h^{-1}\in\Lambda_\Omega^+$. In the first case set
  \[
    \kappa_j(\sigma,g_2)
          = \kappa_{j-1}(\sigma, g_1) \alpha(h, \sigma_1(e)),
  \]
  and otherwise set
  \[ 
    \kappa_j(\sigma,g_2)
        = \kappa_{j-1}(\sigma, g_1)
                \bigl[\alpha(h^{-1}, \sigma_2(e))\bigl]^{-1}.
  \]
  Let $K_j$ be a subgraph of $\cF_n$ with the set of vertices 
  $\kappa_j(\widetilde{\mF}^{\leq j})$. Define 
  $T_\sigma = \bigcup_{j=0}^\infty K_j$. The rest of the proof proceeds 
  similarly to that of Theorem~\ref{thm-main}.
\end{proof}

\begin{remark}
  The main difference in the geometry of graphs in the images of embeddings of 
  Bernoulli shifts, and embeddings of expansive compactly generated pseudogroups 
  is that in the former case vertices of the graphs have constant degree, that 
  is, the number of edges starting and ending at a vertex in $T_\sigma$ is the 
  same for every vertex. In the latter it may vary, with the highest degree 
  always $\card(\Lambda_\Omega^+)$.
\end{remark}


\section{Bernoulli shifts on semigroups}\label{semigroup}

In this section we explore the case of Bernoulli shifts on semigroups. Our 
constructions can be applied to the Bernoulli shift on $\mN_0$, just because its 
dynamics is given by a pseudogroup of transformations generated by finitely many 
transformations of compact domain: if we fix $\Sigma=\Sigma(\mN_0,S)$, we have 
the compact cylinders $C_s = \{\, \sigma\in\Sigma \st \sigma(0)=s \,\}$ and the 
maps
\[
  1_s : C_s \to \Sigma : \sigma \mapsto \sigma \cdot 1,
\]
that is the action of $1\in\mN_0$ restricted to $C_s$, with $s\in S$. It is 
straightforward to show $\{C_s\}_{s\in S}$ is a partition of $\Sigma$ and the 
pseudogroup associated to the action of $\mN_0$ on $\Sigma$ is generated by the 
family $\{1_s\}_{s\in S}$. Therefore it is possible apply the 
Theorem~\ref{cor-pseudgroupuniv} and equivariantly embed $\Sigma$ in the 
space of pointed trees. 
But it is not always possible to extend those arguments to other semigroups. 

The following argument gives a family of semigroups which cannot be 
equivariantly embedded in any foliated space, as their dynamics is not given by a 
pseudogroup of local homeomorphisms. This class of semigroups contains, for instance $\mN_0\times\mN_0$.

\begin{lemma}
  Let $G$ be a semigroup generated by a finite family of 
  elements $G^1$, and fix a finite set of symbols $S$ of cardinality greater 
  than $1$. Then if there exists $a$ and $b\in G$ such that 
  $\langle a\rangle\cong\mN_0$ and $a^k\neq bw$ for any $w\in G$ and 
  $k\in\mN_0$, then the action of $G$ on $\Sigma(G,S)$ cannot be conjugate to 
  the pseudogroup of local homeomorphisms.
\end{lemma}

\begin{proof}
  Suppose there exist $a, b\in G$ as in the hypothesis. Fix an element $\sigma 
  \in\Sigma(G,S)$ and a symbol $s_0\in S$. Recall that the set 
  $\Sigma(\langle a\rangle,S)$ contains an uncountable number of maps. Now for 
  each $\varsigma\in\Sigma(\langle a\rangle,S)$ we will construct 
  $\sigma_\varsigma$ such that $\sigma_\varsigma\cdot b = \sigma$. For that 
  write
  \[
    \sigma_\varsigma(w) = \begin{cases}
      \varsigma(w), & \text{if $w = a^k$ for some $k\in\mN$,} \\
      \sigma(w^\flat), & \text{if $w = b w^\flat $ for some $w^\flat\in G$,} \\
      s_0, & \text{in other case.}
    \end{cases}
  \]
  As $a^k\neq b w$ for any $w\in G$ and $k\in\mN_0$, the map 
  $\sigma_\varsigma:G\to S$ is well defined. Now, by definition
  \[
    \sigma_\varsigma(w)\cdot b = \sigma_\varsigma(bw) = \sigma(w).
  \]
  Clearly if $\varsigma\neq\varsigma'$, then 
  $\sigma_\varsigma\neq\sigma_{\varsigma'}$, i.e. all $\sigma_{\varsigma}$ are 
  distinct. Therefore, there are uncountably many elements of $\Sigma(G,S)$ 
  which are mapped to $\sigma$ by $b\in G$. Since every $T_\sigma$ contains a 
  countable number of vertices, the dynamics on the shift space $\Sigma(G,S)$ 
  cannot be conjugate to the dynamics of $(X_n,\fG_n)$. The same holds for any 
  open set in $\Sigma(G,S)$, as fixing the images of finitely many $a^k$ does 
  not break uncountability of $b^{-1}(\sigma)$. Thus the map associated to the 
  action of $b$ is not locally injective and cannot belong (even locally) to a 
  pseudogroup of local homeomorphisms.
\end{proof}

\end{document}